\documentclass[12pt,a4paper,reqno]{amsart}
\usepackage{geometry}
\usepackage{graphicx}
\usepackage{placeins}
\usepackage[numbers]{natbib}
\usepackage{xcolor}
\usepackage{amsthm,amsmath}
\usepackage{stmaryrd}

\newtheorem{theorem}{Theorem}

\newtheorem{conjecture}[theorem]{Conjecture}
\newtheorem{lemma}[theorem]{Lemma}
\theoremstyle{remark}
\newtheorem*{remark}{Remark}
\newtheorem*{example}{Example}

\numberwithin{equation}{section}

\newcommand{\ii}{i}

\newcommand{\dd}{\mathrm{d}}
\newcommand{\E}{\mathbb{E}}
\newcommand{\RR}{\mathbb{R}}
\newcommand{\CC}{\mathbb{C}}

\begin{document}

\title[Moments of the derivative of zeta]{Complex moments of the derivative of the Riemann zeta function}

\author[C. Hughes]{Christopher Hughes}
\address{Department of Mathematics, University of York, York, YO10 5GH, United Kingdom}
\email{christopher.hughes@york.ac.uk}
\author[A. Pearce-Crump]{Andrew Pearce-Crump}
\address{School of Mathematics, Fry Building, Woodland Road, Bristol, BS8 1UG, United Kingdom}
\email{andrew.pearce-crump@bristol.ac.uk}

\begin{abstract}
We conjecture results about the complex moments of the derivative of the Riemann zeta function, evaluated at the non-trivial zeros of the Riemann zeta function. We do this via two different random matrix computations. In the first, we find an exact formula for the complex moments of the derivative of the characteristic polynomials of unitary matrices averaged over Haar measure using Selberg's integral. In the second, we consider the hybrid approach for zeta, first proposed by Gonek, Hughes and Keating. 
\end{abstract}

\maketitle

\section{Introduction}\label{Sect:Intro}
Let $\zeta (s)$ be the Riemann zeta function. Throughout this paper we assume the Riemann Hypothesis, that is, that all non-trivial zeros can be written as $\rho = \tfrac{1}{2} + i \gamma$ with $\gamma$~real.

The purpose of this paper is to make the following conjecture on discrete moments of the derivative of the Riemann zeta function, which we formulate by developing the philosophy of Keating and Snaith \cite{KS00a}, Hughes, Keating and O'Connell \cite{HKO00}, and Gonek, Hughes and Keating \cite{GHK07} from random matrix theory to discrete moments the zeta function.

\begin{conjecture}\label{conj:n1}
For $\Re(k)>-3$
\begin{equation*}
    \frac{1}{N(T)} \sum_{0 < \gamma \leq T} \zeta' \left( \tfrac{1}{2} + i \gamma \right)^{k} \sim \frac{1}{\Gamma(k +2)} \left( \log \frac{T}{2\pi} \right)^{k},
\end{equation*}
as $T \to \infty$, where $\Gamma (z)$ is the Gamma function and where $N(T)$ is the number of non-trivial zeros in the critical strip, with $0<\gamma\leq T $, given by
\begin{equation*}\label{N(T)}
    N(T) =\frac{T}{2 \pi} \log \frac{T}{2 \pi e} + O(\log T).
\end{equation*}
\end{conjecture}

We prove a number of random matrix results that lead to this conjecture, and they are detailed in Section~\ref{sect:RMTresults}. 

It is worth highlighting here that in order to take complex powers of complex numbers, we need to define the appropriate branch of the logarithm. The way the branch was chosen (discussed in Section~\ref{sect:RMTresults}) is harder to directly apply here without considering the Weierstrass product for zeta. We note that the appropriate choice of branch is not the same as that which comes from continuously varying $\log\zeta'(s)$.

We note that this conjecture agrees with the two previously known cases (ignoring the trivial $k=0$ case): 
\begin{itemize}
    \item The case $k=1$ was proven by Conrey, Ghosh, and Gonek in \cite{CGG88}, unconditionally.
    \item The case $k=-1$ can be found as equation (3.9) in a paper of Garaev and Sankaranarayanan \cite{GarSan06}, under the assumption of simplicity of zeros. It was previously known under the additional assumption of the Riemann Hypothesis, and can be found in various sources, including  \cite{ingham1942two,MV2007,milinovich2012}. 
\end{itemize}

Before forming this general conjecture, no other cases were known or even conjectured.

\begin{remark}
    We make a final remark about the special case when $k=-2$: the right hand side of Conjecture~\ref{conj:n1} is zero. Indeed, in the random matrix calculations that lead to this conjecture, the equivalent calculation equals zero exactly. In the case of zeta, what we mean is the left hand side is bounded by some error terms with no leading-order asymptotic. An educated guess suggests the error terms will be $O(T^{-1/2+\varepsilon})$. 
\end{remark}

\begin{remark}
    In a future paper \cite{HPC}, we will consider the moments of mixed derivatives case, that is, when we consider moments of the form
\[
\sum_{0 < \gamma \leq T} \zeta^{(n_1)} \left( \tfrac{1}{2} + i \gamma \right)\dots\zeta^{(n_k)} \left( \tfrac{1}{2} + i \gamma \right)
\]
for $n_1,\dots,n_k \in \mathbb{N}$, as $T \rightarrow \infty$. In particular, in this integer moment case, we will also give full asymptotic formulae for such sums.  As a special case, we can set all derivative to be equal and consider integer moments of any derivative of the Riemann zeta function given by
\[
\sum_{0 < \gamma \leq T} \zeta^{(n)} \left(\tfrac{1}{2} + i \gamma  \right)^{k}
\]
as $T \rightarrow \infty$ for any $k \in \mathbb{N}$.
\end{remark}

\section{Philosophy of predicting moments of the zeta function through random matrix theory}\label{sect:PreviousRMTresults}
In this section we describe the philosophy of Keating and Snaith \cite{KS00a} and Hughes, Keating and O'Connell \cite{HKO00} from random matrix theory to moments of the zeta function. 

The model proposed by Keating and Snaith \cite{KS00a} is that the characteristic polynomial of a large random unitary matrix can be used to model the value distribution of the Riemann zeta function near a large height $T$. They used the characteristic polynomial
\begin{equation} \label{ZTheta}
    Z(\theta,A) = \det (I - A e^{-i \theta}) = \prod_{n=1}^{N} (1-e^{i(\theta_n - \theta)}),
\end{equation}
where the $\theta_n$ are the eigenangles of an $N \times N$ random unitary matrix $A$, to model $\zeta (s)$. The motivation for studying this is that it has been conjectured that the limiting distribution of the non-trivial zeros of $\zeta (s)$, on the scale of their mean spacing, is asymptotically the same as that of the eigenangles $\theta_n$ of matrices chosen according to Haar measure, in the limit as $N \to \infty$. Equating the mean densities yields the connection between matrix size and height up the critical line
\begin{equation*}
    N = \log \frac{T}{2 \pi}.
\end{equation*}

Averaging over all Haar-distributed $N\times N$ unitary matrices, Keating and Snaith \cite{KS00a} calculated the moments of $|Z(\theta)|$ and found that for fixed real $\theta$, if $\Re(k)>-1/2$, as $N\to\infty$
\begin{equation*}
\mathbb{E}_N \left[ |Z(\theta,A)|^{2k} \right] \sim \frac{G^2 (k+1)}{G(2k+1)} N^{k^2}
\end{equation*}
where $\mathbb{E}_N$ denotes expectation with respect to Haar measure on $N\times N$ unitary matrices, and $G(z)$ is the Barnes $G$-function.

Comparing with all known and conjectured results for continuous moments of zeta, they were lead to conjecture the following result.
\begin{conjecture}[Keating--Snaith]
For any fixed $k$ with $\Re (k) > -1/2$,
\begin{equation*}
    \frac{1}{T} \int_{0}^{T} \left|\zeta \left(\tfrac{1}{2} + it \right) \right|^{2k} \ dt \sim a(k) \frac{G^2 (k+1)}{G(2k+1)} \left( \log \frac{T}{2\pi} \right)^{k^2},
\end{equation*}
as $T \rightarrow \infty$, and
\begin{equation}\label{AF}
    a(k) = \prod_{p \text{ prime}} \left( 1 - \frac{1}{p} \right)^{k^2} \sum_{m=0}^{\infty} \left( \frac{\Gamma (m+k)}{m! \, \Gamma (k)} \right)^2 p^{-m}
\end{equation}
is an arithmetic factor which is discussed below.
\end{conjecture}

The Keating and Snaith result is for a continuous moment, but random matrix theory has also been used to predict discrete moments of $\zeta (s)$. Hughes, Keating and O'Connell~\cite{HKO00} followed a similar approach and calculated the discrete moments of $|Z'(\theta,A)|$ (where the prime denotes differentiation with respect to $\theta$) and found that for $\Re(k)>-3/2$, as $N\to\infty$
\begin{equation}\label{eq:HKOresult}
\mathbb{E}_N \left[ \frac{1}{N} \sum_{n=1}^N |Z'(\theta_n,A)|^{2k} \right] \sim \frac{G^2 (k+2)}{G(2k+3)} N^{k(k+2)}
\end{equation}
where the sum is taken over all the eigenangles of the matrix $A$.

As in the continuous case above, motivated by this result they made a conjecture for zeta which agreed with all known and previously conjectured results and is believed to hold in general.
\begin{conjecture}[Hughes--Keating--O'Connell]
For any fixed $k$ with $\Re (k) > -3/2$,
\begin{equation*}
    \frac{1}{N(T)} \sum_{0 < \gamma \leq T} \left| \zeta' \left( \tfrac{1}{2} + i \gamma \right) \right| ^{2k} \sim a(k) \frac{G^2 (k+2)}{G(2k+3)} \left( \log \frac{T}{2\pi} \right)^{k(k+2)},
\end{equation*}
where $a(k)$ is the same arithmetic factor as in the previous conjecture, given in \eqref{AF}.
\end{conjecture}

What wasn't clear at the time was why the same arithmetic factor \eqref{AF} should appear in these two conjectures. Indeed, the arithmetic factor itself was added in an \textit{ad hoc} manner after the random matrix theory calculations were computed. This additional factor has been shown to be included naturally in the hybrid model developed by Gonek, Hughes and Keating \cite{GHK07} for continuous moments and of Bui, Gonek and Milinovich \cite{BGM13} for discrete moments. This hybrid approach gives equal weighting to $Z_X$, a partial Hadamard product over the zeros of zeta, and $P_X$, a partial Euler product over the primes; it is stated exactly in Theorem~\ref{thm:zeta as product} in the next  section. One strength of the hybrid approach is that it explained why the same arithmetic factor should appear in both conjectures.

\section{Results justifying the formulation of Conjecture \ref{conj:n1}}\label{sect:RMTresults}
The previous two problems where random matrix theory was used both concerned moments of a real function, where an absolute value of zeta was taken. For our problem, we consider the complex moments of the complex function. In this paper we attack our problem in two ways. First, by direct calculation with the characteristic polynomial using Selberg's integral, and then by modelling zeta using the hybrid approach. For the first moment, $k=1$, we will prove that hybrid methodology gives the correct result. 

First, in Section~\ref{sect:Selberg} we prove the following result using Selberg's integral.
\begin{theorem}\label{thm:RMTn1}
Let $Z(\theta,A)$ be as given in \eqref{ZTheta}, and let $\theta_n$ denote the eigenangles of an $N\times N$ unitary matrix. Then for $\Re(k)>-3$ we have
\begin{align*}
\E_N\left[ \frac{1}{N} \sum_{n=1}^N  Z'(\theta_n,A)^k \right] &= e^{i \pi k /2} \frac{\Gamma(N+k+1)}{ N! \Gamma(k+2)}\\
&\sim  \frac{e^{i \pi k /2}}{\Gamma(k +2)} N^{k}
\end{align*}
as $N\to\infty$. 
\end{theorem}

However, in order to use this to create a conjecture for zeta, we need to know what the arithmetic factor equals. To do that, we take the hybrid approach, mentioned above.

In 2007 Gonek, Hughes and Keating \cite{GHK07}, proved a hybrid formula for the zeta function, expressing it as a partial product over primes times a rapidly decaying product over non-trivial zeros. We state a result with slightly more control over the smoothing  (in  \cite{GHK07} and other later papers, $X=Y$ is taken).
\begin{theorem}\label{thm:zeta as product}
     Let $s=\sigma+\ii t$ with $\sigma \geq 0$ and
     $ |t|\geq 2$, let $X,Y \geq 2$ be real parameters, and let $K$ be any fixed
     positive integer.
     Let $f(x)$ be a nonnegative $C^{\infty}$ function of mass 1 supported on $[0, 1]$ and set $u(x)=Y f(Y\log\frac{x}{e}+1)/x$.
     Thus, $u(x)$ is a
     function of mass 1 supported on $[e^{1-1/Y}, e]$.  Set
\begin{equation}\label{U definition}
U(z) = \int_1^e u(y) E_1(z\log y)\;\dd y \,,
\end{equation}
where $E_{1}(z)$ is the exponential integral $\int_{z}^{\infty}
e^{-w}/w\;\dd w $. Then
\begin{equation*}
    \zeta(s) = P_{X}(s) Z_{X}(s)  \left(1+ O\left(\frac{Y^{K+1} X^{\max(1-\sigma,0)}}{(|s|
    \log X)^{K}} \right) + O\left(\frac{X^{1-\sigma}\log X}{Y}\right) \right)\;,
\end{equation*}
where
\begin{equation}\label{eq:PX}
     P_{X}(s) =
   \exp\left(\sum_{n \leq X} \frac{\Lambda(n)}{\log n}\frac{1}{n^{s}} \right)\,,
\end{equation}
$\Lambda(n)$ is von Mangoldt's function, and
\begin{equation}\label{eq:ZX}
   Z_{X}(s) = \exp\left(-\sum_{\rho_n}U\big((s-\rho_n)\log
    X\big)\right)\,.
\end{equation}
The constants implied by the $O$ terms depend only on $f$ and $K$.
\end{theorem}

\begin{remark}
    The proof of this result follows precisely the method in \cite{GHK07}, with obvious changes for the different smoothing.
\end{remark}

Just as in work of Bui, Gonek and Milinovich \cite{BGM13}, we can differentiate whilst still obtaining the asymptotic,  obtaining under the Riemann Hypothesis
\begin{equation}\label{eq:HybridZetaPrimeRho}
\zeta'(\rho) =  P_X(\rho) Z_X'(\rho) \left(1+O\left(\frac{Y^{K+1} X^{1/2}}{(|\rho| \log X)^K} \right) + O\left(\frac{X^{1/2}\log X}{Y} \right) \right)
\end{equation}

Note that $Z_X'(\rho)$ can be expressed as a product
\[
Z'_X(\rho) = C \prod_{\rho_n\neq\rho} \exp\left(-\sum_{\rho_n}U\big((\rho-\rho_n)\log
    X\big)\right)
\]
for $C$ some positive constant. The product structure informs our choice of branch, needed for when the moment $k$ is not an integer, which is to let $U\big(\sigma+(\rho-\rho_n)\log X\big)$ vary continuously for $\sigma>0$ and tending to zero as $\sigma\to+\infty$. This is the equivalent of the choices we make in the random matrix theorems in Sections~\ref{sect:Selberg} and \ref{sect:ZX}.

In Section~\ref{sect:ZX} we model the discrete moments of $Z_X(\rho)$ from the hybrid model by random matrix theory. The appropriate random matrix equivalent for $Z_X(s)$ is
\begin{equation}\label{eq:ZNX}
   Z_{N,X}(\theta,A) = \prod_{n=1}^N\exp\left(-\sum_{j=-\infty}^\infty U\big(\ii(\theta-\theta_n + 2\pi j)\log X\big)\right)
\end{equation}
where the $N\times N$ unitary matrix has eigenvalues $e^{i\theta_n}$ and $U$ is given in \eqref{U definition}.

\begin{remark}
    The sum over $j$ makes the arguments $2\pi$-periodic, thus permitting the move from $e^{i\theta}$ to $\theta$ without any branch-cut ambiguities. The decay of $U$ for large inputs means that the infinite sum over $j$ is absolutely summable.
\end{remark}

In Section~\ref{sect:ZX} we prove the following result for the modelled version of $Z_X'(\rho)$. 
\begin{theorem} \label{thm:Z_X}
    Let $Z_{N,X}(\theta,A)$ be given in \eqref{eq:ZNX}, and let $\theta_n$ denote the eignenangles of an $N\times N$ unitary matrix $A$. If $Re(k) >-3$ then
    \[
    \E_N\left[\frac{1}{N} \sum_{n=1}^N Z_{N,X}'(\theta_n,A)^k\right] \sim  \frac{e^{\ii k \pi/2}}{\Gamma(k+2)} N^k
    \]
    as $N\to \infty$.
\end{theorem}
Comparing the results of this theorem with Theorem~\ref{thm:RMTn1} shows that the product over zeros in the hybrid model is predicted to  have the same leading order asymptotics as the characteristic polynomial.

Finally, in Section~\ref{sect:PX}, we calculate the discrete moments of $P_X(\rho)$ and show that the mean is 1. In particular, we are able to prove the following moment result for $P_X$.
\begin{theorem}\label{thm:PXComplex}
Let $P_X(s)$ be given by \eqref{eq:PX}. Under the Riemann Hypothesis, for $X>2$ with $X=O(\log T)$, for $k\in\CC$ a fixed complex number,
\[
\sum_{0 < \gamma \leq T} P_X(\rho)^k = N(T) + O(T\log\log T)
\]
as $T\to\infty$.
\end{theorem}

It is believed that $P_X(s)$ and $Z_X(s)$ operate pseudo-independently (this is the so-called ``Splitting Conjecture'' of Gonek--Hughes--Keating~\cite{GHK07}), and so the moments of zeta are products of moments of $P_X$ and $Z_X$. The same heuristic applies for discrete moments of $P_X(\rho)$ and $Z_X'(\rho)$ (as discussed in \cite{BGM13}), so multiplying the moments from Theorem \ref{thm:Z_X} and Theorem \ref{thm:PXComplex} together enables us to conjecture that there is no arithmetic term at the leading order for the moments of the complex Riemann zeta function. This is the key difference to the moments when the absolute value of $\zeta'(\rho)$ is taken (as discussed in Section~\ref{sect:PreviousRMTresults}). 

We justify the claims made in the previous paragraph in the case $k=1$ in Theorem~\ref{thm:twist} in Section~\ref{sect:twisted} by applying a twisted first moment for $\zeta'(\rho)$ due to Benli--Elma--Ng~\cite{BEN23}.

Finally, we note that the derivative of $Z(\theta)$ is with respect to $\theta$, and this acts like the derivative of $\zeta(1/2+i t)$ with respect to $t$. This is obviously equal to $i \zeta'(1/2+i t)$, and so we do not expect there to be a factor of $i^k = e^{i \pi k / 2}$ for the Riemann zeta function.

After combining these observations, we form Conjecture \ref{conj:n1}.

\section{Complex moments of the derivative of the characteristic polynomial: Proof of Theorem~\ref{thm:RMTn1}}\label{sect:Selberg}

Differentiating with respect the $\theta$ the characteristic polynomial $Z(\theta,A)$, given in \eqref{ZTheta},  yields
\[
Z'(\theta_r,A) = i \prod_{\substack{n=1 \\ n \neq r}}^{N} \left(1-e^{i(\theta_n - \theta_r)}\right)
\]

The Haar probability density of $U(N)$ equals (see, for example, \cite{Weyl46})
\begin{equation}\label{eq:Weyl}
    d\mu_N = \frac{1}{N! (2 \pi)^N} \prod_{1 \leq j < \ell \leq N} |e^{i\theta_j} - e^{i\theta_\ell}|^2 \ d\theta_1 \ldots d\theta_N.
\end{equation}
so the moment that we wish to calculate is
\begin{multline*}
    \mathbb{E}_N \left[ \frac{1}{N} \sum_{r=1}^{N} Z'(\theta_r,A)^k   \right] \\
    = \frac{e^{i \pi k /2}}{N! (2 \pi)^N} \idotsint_{0}^{2\pi} \prod_{1 \leq j < \ell \leq N} |e^{i\theta_j} - e^{i\theta_\ell}|^2  \frac{1}{N} \sum_{r=1}^N \prod_{\substack{n=1 \\ n\neq r}}^{N} (1-e^{i(\theta_n-\theta_r)})^k \ d\theta_1 \dots d\theta_N
\end{multline*}
where we choose the branch so that $\arg(1-e^{-\varepsilon}e^{i(\theta_n-\theta_r)})$ varies continuously for $\varepsilon>0$, tending to $0$ as $\varepsilon \to \infty$, which is the same as insisting $-\pi/2 < \arg(1-e^{i(\theta_n-\theta_r)}) < \pi/2$. Similarly, where needed, we write $i=e^{i \pi /2}$ and $-i=e^{-i \pi / 2}$.

First note that there are no distinguished eigenvalues, so the average over all eigenanles $\theta_r$ is the same as just the $r=N$ term, say.  By rotation invariance of Haar measure changing variables to $\theta_m \mapsfrom \theta_m - \theta_N$ for $m=1,\dots,N-1$ makes the final $\theta_N$ integral trivial, and we have, after some simple manipulations,
\begin{multline*}
\mathbb{E}_N \left[ \frac{1}{N} \sum_{r=1}^{N} Z'(\theta_r,A)^k   \right]
\\=  \frac{e^{i \pi k /2}}{(2\pi)^{N-1} N!} \idotsint_0^{2\pi} \prod_{1 \leq j < \ell \leq N-1} \left|e^{i\theta_\ell} - e^{i\theta_j} \right|^2 \prod_{n=1}^{N-1}  \left|1-e^{i\theta_n}\right|^2 \left(1-e^{i\theta_n}\right)^k  d\theta_n
\end{multline*}

Using the identities
\begin{align*}
\left|e^{i\theta_\ell} - e^{i\theta_j} \right|^2 &= 4\sin\left(\tfrac12(\theta_\ell - \theta_j)\right)^2 \\
\intertext{and}\\
\left|1-e^{i\theta_n}\right|^2 \left(1-e^{i\theta_n}\right)^k  &= (- e^{i\pi /2})^k e^{i k \theta_n/2} 2^{k+2} \sin\left(\tfrac12 \theta_n\right)^{k+2}
\end{align*}
the expectation we desire equals
\begin{multline*}
e^{-i\pi k(N-2)/2} \frac{2^{(N-2)(N-1) + (k+2)(N-1)}}{(2\pi)^{N-1} N!} \\
\idotsint_0^{2\pi} \prod_{1 \leq j < \ell \leq N-1} \sin\left(\tfrac12(\theta_\ell - \theta_j)\right)^2  \prod_{n=1}^{N-1}  e^{i k \theta_n/2} \sin\left(\tfrac12 \theta_n\right)^{k+2}  d\theta_n .
\end{multline*}
We now follow the train of changes of variables utilised by Keating and Snaith in \cite{KS00a}, when they evaluated the complex moments of the absolute value of the characteristic polynomial. First we let $\theta_n/2 \mapsto \theta_n$ to obtain
\begin{multline*}
e^{-i\pi k(N-2)/2} \frac{2^{(N-2)(N-1) + (k+2)(N-1)+(N-1)}}{(2\pi)^{N-1} N!} \\
 \idotsint_0^{\pi} \prod_{1 \leq j < \ell \leq N-1} \sin\left(\theta_\ell - \theta_j\right)^2  \prod_{n=1}^{N-1}  e^{i k \theta_n} \sin\left(\theta_n\right)^{k+2}  d\theta_n
\end{multline*}
The next step is to observe that
\[
\sin(\theta_\ell - \theta_j) = \sin \theta_\ell  \sin \theta_j \left(\cot \theta_\ell  - \cot \theta_j \right)
\]
so we have
\begin{multline*}
e^{-i\pi k(N-2)/2} \frac{2^{(N-2)(N-1) + (k+2)(N-1)+(N-1)}}{(2\pi)^{N-1} N!} \\
 \idotsint_0^{\pi} \prod_{1 \leq j < \ell \leq N-1} \left(\cot \theta_\ell - \cot \theta_j\right)^2  \prod_{n=1}^{N-1}  e^{i k \theta_n} \sin\left(\theta_n\right)^{k+2+2N-4}  d\theta_n
\end{multline*}
Finally we change variables to $x_n = \cot(\theta_n)$, noting that $d x_n= - d\theta_n/\sin^2\theta_n$, and using $(\sin \theta_n)^2 = 1/(1+x_n^2)$ and, for $0 \leq \theta_n \leq \pi$,
\begin{align*}
e^{i\theta_n} \sin \theta_n &= \cos \theta_n \sin\theta_n + i \sin^2 \theta_n \\
&= \frac{x_n + i }{1+x_n^2} \\
&= i \frac{1}{1+i x_n}
\end{align*}
with $\theta=0$ corresponding to $+\infty$ and $\theta=\pi$ corresponding to $-\infty$. This substitution yields
\begin{multline*}
e^{-i\pi k(N-2)/2} \frac{2^{(N-1)(N+k+1)}}{(2\pi)^{N-1} N!} \\
 \idotsint_{-\infty}^{\infty} \prod_{1 \leq j < \ell \leq N-1} \left(x_\ell - x_j\right)^2  \prod_{n=1}^{N-1} e^{i \pi k /2} \left(1+i x_n\right)^{-k}  \left(1+x_n^2\right)^{-N}  d x_n
\end{multline*}
where the minus sign from the change of variables is used to reverse the integrals' order so they run from $-\infty$ to $+\infty$.

Expanding and simplifying, this is
\begin{multline*}
e^{i \pi k /2}  \frac{2^{(N-1)(N+k+1)}}{(2\pi)^{N-1} N!} \\
 \idotsint_{-\infty}^{\infty} \prod_{1 \leq j < \ell \leq N-1} \left(x_\ell - x_j\right)^2  \prod_{n=1}^{N-1} \left(1+i x_n\right)^{-k-N}  \left(1-i x_n\right)^{-N}  d x_n
\end{multline*}

The integral is now exactly a form of Selberg's integral, $J(a,b,\alpha,\beta,\gamma,N)$ (see \cite[equation 17.5.2]{Mehta}, for example\footnote{A closer inspection of the proof in \cite{Mehta} shows the stated regime of convergence in 17.5.5 is not correct, and instead should read $-\frac{1}{N} < \Re(\gamma) < \frac{\Re(\alpha)+\Re(\beta)-1}{2(N-1)}$ for the integral $J(a,b,\alpha,\beta,\gamma,N)$ to converge.}), and we have
\begin{multline*}
\idotsint_{-\infty}^{\infty} \prod_{1 \leq j < \ell \leq N-1} \left(x_\ell - x_j\right)^2  \prod_{n=1}^{N-1} \left(1+i x_n\right)^{-k-N}  \left(1-i x_n\right)^{-N}  d x_n \\
=
J(1,1,k+N,N,1,N-1) \\
= \frac{(2\pi)^{N-1}}{2^{(N-1)(k+N+1)}} \prod_{j=0}^{N-2} \frac{\Gamma(2+j) \Gamma(k+N+1-j)}{\Gamma(k+N-j)\Gamma(N-j)}
\end{multline*}
with the integral on the left-hand side converging so long as $\Re(k)>-3$, and so the expectation becomes
\begin{align*}
\E_N\left[ Z'(\theta_N,A)^k \right] &= e^{i \pi k /2} \frac{1}{ N!} \prod_{j=0}^{N-2} \frac{\Gamma(2+j) \Gamma(k+N+1-j)}{\Gamma(k+N-j)\Gamma(N-j)}\\
&=e^{i \pi k /2} \frac{\Gamma(N+k+1)}{ N! \Gamma(k+2)} \\
&\sim  \frac{e^{i \pi k /2}}{\Gamma(k +2)} N^{k}
\end{align*}
for large $N$, completing the proof of Theorem~\ref{thm:RMTn1}. Note that the right-hand side provides an analytic continuation of the left-hand side beyond where the expectation naively makes sense. However, we wish to use the characteristic polynomial as a model for the Riemann zeta function, so will restrict ourselves to regions where the expectation converges.

\section{The moments of $Z_X'$: Proof of Theorem~\ref{thm:Z_X}} \label{sect:ZX}

Recall the definition of $Z_{N,X}(\theta,A)$, given in \eqref{eq:ZNX},
\[
Z_{N,X}(\theta,A) = \prod_{n=1}^N\exp\left(-\sum_{j=-\infty}^\infty U\big(\ii(\theta-\theta_n + 2\pi j)\log X\big)\right)
\]
where the $\theta_n$ are the eigenangles of the unitary matrix $A$ and where
\[
U(z) = \int_1^e u(y) E_1(z\log y)\;\dd y \, .
\]
Note that $U(z)$ is not an analytic function; it has a logarithmic singularity at $z=0$. To see this, recall the formula
\begin{equation*}
E_1(z) = - \log z  -\gamma  - \sum_{m=1}^{\infty}
\frac{(-1)^{m}z^{m}}{m!\,m} \,,
\end{equation*}
where $|\arg z| < \pi$, $\log z$ denotes the principal
branch of the logarithm, and $\gamma$ is Euler's constant.  From this and \eqref{U definition} we observe that we may interpret
$\exp(-U(z))$ to be an analytic function, asymptotic to $C z$ as $z\to 0$, for some constant $C$ which depends upon the smoothing function $u$.

It is convenient to factor out the behaviour around the eigenvalues (where $Z_{N,X}$ vanishes) as follows
\[
Z_{N,X}(\theta, A) = \prod_{m=1}^N \left(1-e^{-\ii(\theta-\theta_m)} \right) e^{F_X(\theta-\theta_m)}
\]
where
\begin{equation}\label{eq:Ftheta}
F_X(\vartheta) =  - \log\left(1-e^{-\ii \vartheta} \right) -\sum_{j=-\infty}^\infty U\big(\ii(\vartheta + 2\pi j)\log X\big)
\end{equation}
is a $2\pi$-periodic function which is continuously differentiable for all real $\vartheta$ --- the logarithmic singularities cancel out. The first statement is obvious. The second statement can be seen from the definition of $U(z)$ and the series expansion of $E_1(z)$ around $z=0$, but will be made explicit when we look at the decay of its Fourier coefficients in Lemma~\ref{lem:FourierCoeffs}).

Differentiating with respect to $\theta$,
\begin{multline*}
Z_{N,X}'(\theta,A) = \sum_{m=1}^N \ii e^{-\ii(\theta-\theta_m)} e^{F_X(\theta-\theta_m)} \prod_{\substack{n=1\\n \neq m}}^N \left(1-e^{-\ii(\theta-\theta_n)} \right) e^{F_X(\theta-\theta_n)} \\
+ \sum_{m=1}^N \left(1-e^{-\ii(\theta-\theta_m)} \right) F_X'(\theta-\theta_m)  e^{F_X(\theta-\theta_m)} \prod_{\substack{n=1\\n \neq m}}^N \left(1-e^{-\ii(\theta-\theta_n)} \right) e^{F_X(\theta-\theta_n)}
\end{multline*}
and substituting $\theta=\theta_N$, we see that for every $m \neq N$ there is a term in the product that vanishes (so only the $m=N$ terms survives), and every term in the second term vanishes, including the $m=N$ term. That is,
\[
Z_{N,X}'(\theta_N,A) = \ii  e^{F_X(0)} \prod_{n=1}^{N-1} \left(1-e^{-\ii(\theta_N-\theta_n)} \right) e^{F_X(\theta_N-\theta_n)}
\]

\begin{remark}
    There is nothing special about picking $\theta_N$ over any of the other eigenvalues. This is made for notational convenience. The true calculation would be to evaluate
    \[
    \frac{1}{N} \sum_{n=1}^N Z_{N,X}'(\theta_n,A)
    \]
    similar to what we do in Section~\ref{sect:Selberg}, but by rotation invariance of Haar measure the result will be the same.
\end{remark}

We now calculate the expected value of $Z_{N,X}'(\theta_N,A)^k$ for $k\in\CC$ when averaged over all $N\times N$ unitary matrices $A$ distributed according to Haar measure, employing the same trick we used in Section~\ref{sect:Selberg} (which was first employed in \cite{HKO00} for similar purposes). We have
\begin{align}
\E_{N} \left[ Z_{N,X}'(\theta_N,A)^k \right] &=  e^{\ii k \pi / 2} e^{k F_X(0)} \E_N\left[ \prod_{n=1}^{N-1} \left(1-e^{-\ii(\theta_N-\theta_n)} \right)^k e^{k F_X(\theta_N-\theta_n)}  \right] \notag\\
&= e^{\ii k \pi / 2} e^{k F_X(0)} \frac{1}{N} \E_{N-1} \left[ \prod_{n=1}^{N-1} \left|1-e^{\ii\vartheta_n}\right|^2  \left(1-e^{\ii \vartheta_n} \right)^k e^{k F_X(-\vartheta_n)} \right] \label{eq:MmtsZNX}
\end{align}
where we interpret $e^{\ii \vartheta_n}$ as the eigenvalues of an $(N-1) \times (N-1)$ unitary matrix. (Briefly, what we have done is write the first expectation out as a $N$-dimensional Weyl integral, then change variables to $\vartheta_n = \theta_n - \theta_N$ for $n=1,\dots,N-1$. Making use of the fact the integrand is $2\pi$-periodic, we turn it back into a ($N-1$)-dimension Weyl integral, plus an extra trivial integral over $\theta_N$).

Because $k$ is complex, we need to be precise about the choice of branch taken. We define $\left(1-e^{\ii \vartheta} \right)^k$ to be the value obtained from continuous variation of  $\left(1-e^{\ii \vartheta}e^{-\varepsilon} \right)^k$ for $\varepsilon>0$, converging to the value $1$ as $\varepsilon\to\infty$. This is the same as our choice of branch for the characteristic polynomial in Section~\ref{sect:Selberg}.

We may now use Heine's Lemma~\cite{Sze39}\footnote{Though the identity was first written down in 1939 by Szeg\H{o}, he gave the credit to Heine (who lived from 1821 to 1881).} to write the inner expectation as a Toeplitz determinant. We find that
\begin{align*}
\E_{N-1} \left[ \prod_{n=1}^{N-1} \left|1-e^{\ii\vartheta_n}\right|^2  \left(1-e^{\ii \vartheta_n} \right)^k e^{k F_X(-\vartheta_n)} \right] &= D_{N-1}[f] \\
&:= \det\left\vert \hat f_{j-\ell} \right|_{1 \leq j,\ell \leq N-1}
\end{align*}
with
\[
\hat f_{j-\ell} = \frac{1}{2\pi} \int_0^{2\pi} f(\vartheta) e^{-\ii (j-\ell) \vartheta} \dd \vartheta
\]
for the symbol
\begin{equation*}
f(\vartheta) = \left|1-e^{\ii\vartheta}\right|^2  \left(1-e^{\ii \vartheta} \right)^k e^{k F_X(-\vartheta)}
\end{equation*}

Note that $f(\vartheta)$ has only one singularity at $\vartheta=0$. Toeplitz determinants with such symbols have been calculated asymptotically by Ehrhardt and Silbermann in~\cite{ErhSil97} (building on extensive previous work by several other authors).

They show (in their Theorem 2.5) that if $b(\vartheta)$ is a suitably smooth (defined in their paper) $2\pi$-periodic function with winding number zero and whose logarithm has the Fourier series expansion
\[
\log b(\vartheta) = \sum_{m=-\infty}^\infty s_m e^{\ii m \vartheta}
\]
then for symbols of the form
\[
f(\vartheta) = b(\vartheta) \left(1-e^{\ii \vartheta}\right)^\gamma \left(1-e^{-\ii \vartheta}\right)^\delta
\]
subject to $\gamma+\delta \not\in\{-1,-2,-3,\dots\}$, then for any $\varepsilon>0$
\begin{equation*}
    D_{N-1}[f] = C_1 N^{\gamma \delta} C_2^{N-1} \left(1+ O\left(\frac{1}{N^{1-\varepsilon}}\right) \right)
\end{equation*}
where
\begin{equation*}
C_1 = \frac{G(1+\gamma) G(1+\delta)}{G(1+\gamma+\delta)} \exp\left(\sum_{m=1}^\infty m s_m s_{-m} \right) \exp\left(-\delta \sum_{m=1}^\infty s_m \right)   \exp\left(-\gamma \sum_{m=1}^\infty s_{-m} \right)
\end{equation*}
where $G$ is the Barnes $G$-function, and where
\[
C_2 = \exp(s_0) .
\]

In our case, $\gamma=k+1$ and $\delta=1$ and $b(\vartheta) = \exp(k F_X(-\vartheta))$. To complete the proof, we simply need to calculate the Fourier coefficients of $k F_X(-\vartheta)$ (those coefficients being the desired $s_m$).

We will, in fact, delay the exact calculation of $s_m$ to Lemma~\ref{lem:FourierCoeffs} to quickly complete the proof of the theorem. All we need from the lemma is that $s_m=0$ if $m \leq 0$. From \eqref{eq:MmtsZNX} and using Ehrhardt and Silbermann's formula above, with the values $\gamma=k+1$ and $\delta=1$, we have
\[
\E_{N} \left[ Z_{N,X}'(\theta_N,A)^k \right] \sim  e^{\ii k\pi/2} e^{k F_X(0)} \frac{1}{N} \frac{G(k+2) G(2)}{G(k+3)} \exp\left(-\sum_{m=1}^\infty s_m\right) N^{k+1}
\]

Finally, note that $G(2)=1$, $G(k+3) = \Gamma(k+2) G(k+2)$ and
\[
e^{k F_X(0)} = \exp\left(\sum_{m=-\infty}^\infty s_m \right) = \exp\left(\sum_{m=1}^\infty s_m \right)
\]
(the first equality following from setting $\vartheta=0$ in the Fourier series expansion; the second equality comes from knowing $s_m=0$ for $m\leq 0$). Therefore,
\[
\E_{N} \left[ Z_{N,X}'(\theta_N,A)^k \right] \sim   \frac{e^{\ii k \pi/2}}{\Gamma(k+2)} N^k
\]
as $N\to\infty$, as required. This completes the proof of Theorem~\ref{thm:Z_X}.

\begin{remark}
    The proof of this theorem  deals with the analytic continuation of the objects under consideration, and the result holds for $k \in \CC$ such that $k \not\in \{-3,-4,-5,\dots\}$. Similar to the restriction imposed in Theorem~\ref{thm:RMTn1}, for our purposes as a model for zeta, we must stop before the first singularity; that is, in Conjecture~\ref{conj:n1} we take $\Re(k)>-3$.
\end{remark}

\begin{lemma} \label{lem:FourierCoeffs}
    Let $F_X(\vartheta)$ be defined in \eqref{eq:Ftheta}, then
    \[
    k F_X(-\vartheta) = \sum_{m=-\infty}^\infty s_m e^{\ii m \vartheta}
    \]
    where
    \[
    s_m =
    \begin{cases}
    0 & \text{ if } m \leq 0\\
    \frac{k}{m}  \left(1 - \int_1^{\exp(m/\log X)} u(y) \dd y \right) & \text{ if } 1 \leq m < \log X  \\
    0 & \text{ if } m\geq \log X
    \end{cases}
    \]
\end{lemma}

\begin{remark}
    Since $u$ has total mass 1, for $m>0$ we can write
    \[
    1 - \int_1^{\exp(m/\log X)} u(y) \dd y = \int_{\exp(m/\log X)}^\infty u(y) \dd y
    \]
    although due to the support condition on $u$, if $m\geq \log X$ then $\exp(m/\log X)>e$ and so the integral vanishes.
\end{remark}

\begin{proof}[Proof of Lemma~\ref{lem:FourierCoeffs}]
Note that
\[
s_m = \frac{1}{2\pi} \int_{-\pi}^\pi k F_X(-\vartheta) e^{-\ii m \vartheta} \dd \vartheta .
\]
From \eqref{eq:Ftheta} we have
\[
k F_X(-\vartheta) =  - k\log\left(1-e^{\ii \vartheta} \right) -k \sum_{j=-\infty}^\infty U\big(\ii(-\vartheta + 2\pi j)\log X\big) .
\]
For $m$ an integer, we have
\begin{equation}\label{eq:FourierCoeffLog}
\frac{1}{2\pi} \int_{-\pi}^\pi -k \log\left(1-e^{\ii \vartheta} \right) e^{-\ii m \vartheta} \dd\vartheta =
\begin{cases}
    0 & \text{ if } m \leq 0\\
    \frac{k}{m} & \text{ if } m \geq 1
\end{cases}
\end{equation}
(To non-rigorously see why, note that $\log\left(1-e^{\ii \vartheta} \right) = -\sum_{\ell=1}^\infty e^{\ii \ell \vartheta}/\ell$ and the only term in the sum that survives the integration is when $\ell=m$ for positive integers $m$).

Furthermore, we have
\begin{multline*}
\frac{1}{2\pi}  \int_{-\pi}^\pi -k \sum_{j=-\infty}^\infty U\left(\ii(-\vartheta+2\pi j) \log X \right) e^{-\ii m \vartheta} \dd\vartheta
= \frac{-k}{2\pi}
\int_{-\infty}^\infty U(-\ii \vartheta \log X) e^{-\ii m \vartheta} \dd\vartheta\\
= \frac{-k}{2\pi} \int_{-\infty}^\infty \int_1^e u(y) E_1(-\ii \vartheta \log X\log y) e^{-\ii m \vartheta} \;\dd y \dd\vartheta
\end{multline*}
where the final equality comes from inserting the definition of $U$ from \eqref{U definition} and the support of $u$. Swapping the order of integration, this is
\[
\frac{-k}{2\pi}  \int_1^e u(y) \int_{-\infty}^\infty E_1(-\ii \vartheta \log X\log y) e^{-\ii m \vartheta} \; \dd\vartheta \dd y .
\]

The Fourier transform of the exponential integral is
\[
\frac{1}{2\pi} \int_{-\infty}^\infty E_1(-\ii \vartheta \log X\log y) e^{-\ii m \vartheta} \; \dd\vartheta  =
\begin{cases}
    0 & \text{ if } m<\log X \log y\\
    \frac{1}{2m} & \text{ if } m = \log X \log y\\
    \frac{1}{m} & \text{ if } m > \log X \log y
\end{cases}
\]

Therefore, making use of the fact that $u$ is supported on $[1,e]$ and has total mass 1, we have
\begin{multline*}
\frac{-k}{2\pi}  \int_1^e u(y) \int_{-\infty}^\infty E_1(-\ii \vartheta \log X\log y) e^{-\ii m \vartheta} \; \dd\vartheta \dd y = \frac{-k}{m} \int_1^{\max(1,\exp(m/\log X))} u(y) \dd y \\
=\begin{cases}
    0 & \text{ if } m \leq 0 \\
    -\frac{k}{m} \int_1^{\exp(m/\log X)} u(y) \dd y & \text{ if } 1 \leq m < \log X \\
    -\frac{k}{m}  & \text{ if } m \geq \log X
\end{cases}
\end{multline*}

The Fourier coefficient, $s_m$ is the sum of this and \eqref{eq:FourierCoeffLog}. Note that both terms are zero for $m\leq 0$ and the two terms perfectly cancel each other if $m\geq \log X$.
\end{proof}

\section{The moments of $P_X$: Proof of Theorem~\ref{thm:PXComplex}}\label{sect:PX}

Recall $P_X(s)$, given by \eqref{eq:PX}. In this section we will show that for any complex $k$, the mean of $P_X(\rho)^k$ is asymptotically 1, when averaged over zeros of the zeta function. Specifically Theorem~\ref{thm:PXComplex} says that under the Riemann Hypothesis, for $X>2$ with $X=O(\log T)$
\[
\frac{1}{N(T)} \sum_{0 < \gamma \leq T} P_X(\rho)^k = 1 + O\left(\frac{\log\log T}{\log T}\right)
\]
as $T\to\infty$.

This should be compared to a result of Bui, Gonek and Milinovich (Theorem 2.2 of \cite{BGM13}) which stated that, under the Riemann Hypothesis,
for $\varepsilon > 0$ and $X, T \to\infty$ with $X=O\left( (\log T)^{2-\varepsilon} \right)$ then for any $k\in\RR$ one has
\[
\frac{1}{N(T)} \sum_{0 < \gamma \leq T} \left\vert P_X(\rho)\right\vert^{2k} = a(k) \left(e^{\gamma_0} \log X\right)^{k^2} \left(1+O_k\left((\log X)^{-1} \right) \right)
\]
where $a(k)$ is an explicit product over primes given by \eqref{AF}.

Before we prove Theorem~\ref{thm:PXComplex} we need to state and prove some preliminary lemmas.

\begin{lemma}\label{lem:DirichletExpansionP_X}
For any $k \in \CC$, the function $P_X(s)^k$ can be written as a non-vanishing absolutely convergent Dirichlet series
\[
P_X(s)^k = \sum_{m=1}^\infty \frac{a_k(m)}{m^s}
\]
with $a_k(1)=1$, and $|a_k(m)| \leq d_{|k|}(m)$ where $d_{|k|}(m)$ is the $|k|$\textsuperscript{th} divisor function, and with 
\[
a_k(p)=
\begin{cases}
    k & \text{ if } p\leq X\\
    0 & \text{ if } p>X .
\end{cases}
\]
\end{lemma}

\begin{proof}
Since $P_X(s)$ given in \eqref{eq:PX} is the exponential of a finite Dirichlet polynomial, its $k$th power is an entire non-vanishing function, so can be written as a Dirichlet series
\begin{align*}
P_X(s)^k &= \exp\left(k \sum_{n\leq X}  \frac{\Lambda(n)  }{\log n}\frac{1}{n^s} \right)\\
&= \sum_{m=1}^\infty \frac{a_k(m)}{m^s}
\end{align*}
where the sum converges absolutely for any $s$. Note that, by construction, $a_k$ is multiplicative (indeed, it is the multiplicative nature of $P_X$ that allows us to take complex moments).

Temporarily letting $z=p^{-s}$, if we define $\ell$ be to be largest integer such that $p^\ell \leq X$ then we have
\begin{equation}\label{eq:a_k_exp}
\sum_{r=0}^\infty a_k(p^r) z^r = \exp\left(\sum_{j=1}^\ell \frac{k}{j} z^j \right)
\end{equation}
where the left-hand side essentially comes from expanding out the exponential as a series around $z=0$.

For example, if $p\leq X$ is a prime (which means $\ell \geq 1$), then $a_k(p)=k$, and if $p>X$ then $a_k(p)=0$. In particular, this implies $a_k(m)=0$ if $m$ has any prime divisor greater than $X$.

In general $a_k(p^r)$ can be written as a polynomial in $k$ of degree $r$, and $a_k(p^r) = d_k(p^r)$ so long as $p^r \leq X$ (that is, so long as $r\leq \ell$). If $p^r > X$ then due to missing terms in the exponential, the coefficients in the polynomial will be strictly smaller. Hence $|a_k(p^m)| \leq d_{|k|}(p^m)$ for any complex $k$, any prime $p$, any power $r$.
\end{proof}

\begin{example}
To demonstrate the last inequality more explicitly, consider $a_k(p^4)$. Expanding the right-hand side of \eqref{eq:a_k_exp} for $\ell=4,3,2,1$ we see
\[
a_k(p^4) = 
\begin{cases}
    \frac{k^4}{24}+\frac{k^3}{4}+\frac{11 k^2}{24}+\frac{k}{4} & \text{ if } p^4 \leq X \\ 
    \frac{k^4}{24}+\frac{k^3}{4}+\frac{11 k^2}{24}  & \text{ if } p^3 \leq X < p^4\\
    \frac{k^4}{24}+\frac{k^3}{4}+\frac{k^2}{8}  & \text{ if } p^2 \leq X < p^3\\
   \frac{k^4}{24}  & \text{ if } p \leq X < p^2
\end{cases}
\]
If $\ell >4$ none of the additional terms in the sum on the right-hand side of \eqref{eq:a_k_exp} contribute to $a_k(p^4)$ since the additional terms are $z^5$ or higher powers. In all cases therefore, we have
\[
|a_k(p^4)| \leq  \frac{|k|^4}{24}+\frac{|k|^3}{4}+\frac{11 |k|^2}{24}+\frac{|k|}{4} = \frac{1}{24} \left\lvert k\right\rvert \left(\left\lvert k\right\rvert+1\right) \left(\left\lvert k\right\rvert+2\right) \left(\left\lvert k\right\rvert+3\right) = d_{\lvert k \rvert}(p^4)
\]
\end{example}

In the proof of Theorem~\ref{thm:PXComplex}, it will be important to truncate the infinite sum.
\begin{lemma}[Gonek--Hughes--Keating]\label{lem:TruncatePX}
For $2 \leq X \ll (\log T)^{2-\varepsilon}$ with $\varepsilon>0$, $k \in \CC$, and $\theta$ a small positive number, we have
\[
P_X(\tfrac12 + \ii t)^k = \sum_{m \leq T^{\theta}} \frac{a_k(m)}{m^{1/2 + \ii t}} + O_k\left(T^{-\varepsilon \theta / 2}\right)
\]
\end{lemma}

\begin{proof}
    This follows immediately from the proof of Lemma 2 in \cite{GHK07}, adapted for the case $k$ complex.
\end{proof}

In \cite{Gon93} Gonek proved unconditionally a uniform version of Landau's formula concerning sums over the non-trivial zeros of the Riemann zeta function, including the following corollary which requires the assumption of the Riemann Hypothesis.
\begin{lemma}[Gonek]\label{lem:LandauGonek}
Under the Riemann Hypothesis, for $T>1, m \in \mathbb{N}$ with $m\geq 2$ and $\rho$ a non-trivial zero of the Riemann zeta function $\zeta (s)$,
\begin{equation*}
\sum_{0< \gamma \leq T} m^{-\rho} = - \frac{T}{2 \pi} \frac{ \Lambda (m)}{m} + O(\log (2mT) \log\log (3m)),
\end{equation*}
where  $\Lambda (m)$ is the von Mangoldt function.
\end{lemma}

\begin{remark}
    Note this result does not apply in the case when $m=1$, when the sum over zeros trivially equals $N(T)$.
\end{remark}

\begin{proof}[Proof of Theorem~\ref{thm:PXComplex}]
In the case when $\Re(s)=1/2$ and $X=(\log T)^{2-\varepsilon}$, using Lemmas~\ref{lem:DirichletExpansionP_X} and~\ref{lem:TruncatePX} we have
\[
P_X(s)^k = 1 +  \sum_{2 \leq m \leq T^{\theta}} \frac{a_k(m)}{m^{s}} + O_k\left(T^{-\varepsilon \theta / 2}\right)
\]

Therefore, assuming the Riemann Hypothesis, by  Lemma~\ref{lem:LandauGonek} we have
\begin{align*}
\sum_{0<\gamma\leq T} P_X(\rho)^k &= N(T) + \sum_{0<\gamma \leq T} \sum_{2\leq m\leq T^\theta} \frac{a_k(m)}{m^\rho} + O\left(N(T) T^{-\varepsilon \theta / 2}\right) \\
&= N(T) - \frac{T}{2 \pi} \sum_{2\leq m \leq T^\theta} \frac{ a_k(m) \Lambda (m)}{m} +  O\left(\log T \log\log T \sum_{2\leq m \leq T^\theta} |a_k(m)| \right)
\end{align*}

To bound the first sum on the right-hand side, note that the $\Lambda(m)$ forces the sum to be over primes and prime powers only, and the $a_k(m)$ ensures we only need primes $\leq X$ (since otherwise $a_k(m)=0$). Since $a_k(m) \ll m^\varepsilon$, the square- and higher-powers of primes form a convergent sum. Using Lemma~\ref{lem:DirichletExpansionP_X}  to evaluate $a_k(p)$ we see that if $X < T^{\theta}$
\begin{align*}
\sum_{2\leq m \leq T^\theta} \frac{ a_k(m) \Lambda (m)}{m} &= \sum_{p < X} \frac{ k \log p}{p} + O(1) \\
&\ll \log X
\end{align*}
For the second sum, in the error term, note that since $|a_k(m)| \leq d_{|k|}(m)$ we have
\begin{align}
\sum_{m\leq T^\theta} |a_k(m)| &\leq \sum_{m=1}^\infty d_{|k|}(m) \notag\\
&= P_X(0)^{|k|} \notag\\
&=O \left( \exp\left(c' \frac{X }{\log X} \right) \right) \label{eq:P_X(0)^k}
\end{align}
for any fixed $c'>|k|$. The last line follows from
\begin{align*}
\log P_X(0) &= \sum_{n\leq X}  \frac{\Lambda(n)  }{\log n} \\
&= \frac{X}{\log X} + O\left(\frac{X}{(\log X)^2}\right)
\end{align*}
by the Prime Number Theorem.

Therefore
\[
\sum_{0<\gamma\leq T} P_X(\rho)^k = N(T) + O\left(T \log X\right) + O\left(\log T \log\log T \exp\left(c' \frac{X }{\log X} \right)  \right)
\]

The error terms can be made to approximately match when $X=O(\log T)$.

This completes the proof of Theorem~\ref{thm:PXComplex}.
\end{proof}

\begin{remark}
    In the proof we can take $X$ to be anything up to $o(\log T \log\log T)$ before the second error term dominates the main term of $\frac{1}{2\pi} T\log T$. Since the error term as we have written it comes from a subsidiary main term, we cannot reduce the error down to  $O(T)$.
\end{remark}

\section{Twisted first moment of zeta and the splitting conjecture}\label{sect:twisted}

In this section we show that a twisted first moment due to Benli, Elma, Ng \cite{BEN23} can be used to prove that the hybrid methodology developed in Sections~\ref{sect:RMTresults} and \ref{sect:ZX} is correct for $k=1$, that is, for the first moment.

We have already argued that $Z_X'(\rho)^k$ can be modelled by $Z_{N,X}'(\theta_n,A)^k$, and the moments of these are found in Theorem \ref{thm:Z_X}. We shall now directly calculate $Z_X'(\rho)$ and note that the result is consistent with Theorem~\ref{thm:Z_X} once we recall that we use $N=\log(T/2\pi)$ to translate from our model to zeta. Incidentally, this also shows that, at least for $k=1$, the splitting conjecture discussed at the end of Section~\ref{sect:RMTresults} holds.

By \eqref{eq:HybridZetaPrimeRho} we see that, up to small error, $Z_X'(\rho) = \zeta'(\rho) P_X(\rho)^{-1}$. We prove the following theorem.
\begin{theorem}\label{thm:twist}
    Under the Riemann Hypothesis, for $X>2$ with $X=O(\log T)$
    \[
    \frac{1}{N(T)} \sum_{0<\gamma<T} \zeta'(\rho) P_X(\rho)^{-1} \sim \frac{1}{2} \log T
    \]
    as $T \rightarrow \infty$.
\end{theorem}

\begin{proof}
By equation (27) of Benli, Elma, Ng \cite{BEN23},
    \begin{multline}\label{eq:twist}
    \sum_{0<\gamma<T} \zeta'(\rho) P_X(\rho)^{-1} = \frac{T}{4\pi} \left(\left(\log\frac{T}{2\pi}\right)^2 - 2(1-\gamma_0)\log\frac{T}{2\pi} +2(1-\gamma_0- 3\gamma_1 - \gamma_0^2) \right) \\
    +\frac{T}{2\pi} \sum_{2\leq m\leq N} \frac{a_{-1}(m)}{m} \left(\mathcal{A}_1(1,m) + \mathcal{B}_1(m,T) \right) + O\left(T^{1/2+\varepsilon}\right)
    \end{multline}
where we have specialised their result to $Y(s) \equiv 1$ and where we have 
    \[
    X(s) = P_X(s)^{-1} = \sum_{2\leq m\leq N} \frac{a_{-1}(m)}{m^s}
    \]
with $N=T^\vartheta$ for $\vartheta <1/2$, where the coefficients $a_{-1}(m)$ are given in Lemma~\ref{lem:DirichletExpansionP_X}, and where we have specialised their result to $\mathcal{A}_1(1,m)$ and $\mathcal{B}_1(m,T)$, given by 
\begin{align*}
    \mathcal{A}_1(1,m) &=
    \begin{cases}
        \dfrac{p (\log p)^2}{(p-1)^2} & \text{if } m = p^a \text{ for $p$ prime and } a \in \mathbb{N}, \\
        0 & \text{otherwise},
    \end{cases} 
\end{align*}
and, for primes $p$ and primes $p_1 \neq p_2$, and for natural numbers $a,a_1,a_2$,
\begin{equation*}
\mathcal{B}_1(m,T)= 
        \begin{cases}
        \displaystyle
        -\frac{p}{p - 1} \left( \log p \left( \log\left( \frac{T}{2\pi} \right) - 1 + \gamma_0 \right) + \left(a - \frac{1}{2} \right) \log^2 p \right)
        & \text{if } m = p^a \\[2ex]
        \displaystyle
        \frac{p_1 p_2}{(p_1 - 1)(p_2 - 1)} \log p_1 \log p_2
        &\text{if } m = p_1^{a_1} p_2^{a_2}\\[2ex]
        \displaystyle
        0 & \text{otherwise}.
\end{cases}
\end{equation*}

Note that Lemma~\ref{lem:DirichletExpansionP_X} restricts us to primes $p\leq X$ and $X = O (\log T)$ means that the sum over $m$ on the right-hand side of \eqref{eq:twist} is at worst $O(T \log T \log\log T)$.

Therefore we conclude that for $X = O\left(\log T\right)$ we have,
\[
\sum_{0<\gamma<T} \zeta'(\rho) P_X(\rho)^{-1} = \frac{T}{4\pi} (\log T)^2 + O(T\log T \log\log T)
\]
and since we know this is the same leading order asymptotic for $\zeta'(\rho)$, we have shown both that the splitting conjecture holds in this case, and that the $Z_X'$ term from the hybrid model perfectly captures the leading order asymptotics for zeta.
\end{proof}

\section*{Acknowledgments}
This paper was mostly written during the second author's PhD at the University of York, and various sections appear in his thesis \cite{APCThesis}. Later significant revisions have since been made while the second author was under support from the Heilbronn Institute for Mathematical Research.

\bibliographystyle{amsplain}
%\nocite{*}
\bibliography{bibliography}

\end{document}